\documentclass[11pt]{article}
\usepackage[margin=1in]{geometry}

\usepackage{amsmath, amssymb, amsthm, bm, mathrsfs}
\usepackage{bigints}
\usepackage{mathtools}

\usepackage{graphicx}
\usepackage{caption}
\usepackage{subcaption}
\usepackage{tikz}
\usetikzlibrary{trees, positioning}
\usepackage{float}

\usepackage{booktabs}
\usepackage{multirow, multicol}
\usepackage[inline, shortlabels]{enumitem}

\usepackage{algorithm}
\usepackage{algpseudocode}

\usepackage{natbib}
 \bibpunct[, ]{(}{)}{,}{a}{}{,}%
\usepackage{fix-cm}
\usepackage{indentfirst}
\usepackage{xcolor}
\usepackage[mathscr]{eucal}
\usepackage[none]{hyphenat}

\pgfdeclarepatternformonly{stripes}
{\pgfpointorigin}{\pgfpoint{0.25cm}{0.25cm}}
{\pgfpoint{0.25cm}{0.25cm}}{
  \pgfpathmoveto{\pgfpoint{0cm}{0cm}}
  \pgfpathlineto{\pgfpoint{0.25cm}{0.25cm}}
  \pgfpathlineto{\pgfpoint{0.25cm}{0.12cm}}
  \pgfpathlineto{\pgfpoint{0.12cm}{0cm}}
  \pgfpathclose\pgfusepath{fill}
  \pgfpathmoveto{\pgfpoint{0cm}{0.12cm}}
  \pgfpathlineto{\pgfpoint{0cm}{0.25cm}}
  \pgfpathlineto{\pgfpoint{0.12cm}{0.25cm}}
  \pgfpathclose\pgfusepath{fill}
}

\newcommand{\bc}{\begin{center}}
\newcommand{\ec}{\end{center}}
\newcommand{\beq}{\begin{equation}}
\newcommand{\eeq}{\end{equation}}
\newcommand{\beqa}{\begin{eqnarray*}}
\newcommand{\eeqa}{\end{eqnarray*}}

\newcommand{\vect}[1]{\boldsymbol{#1}}
  
\newcommand{\Pb}{\mbox{{I}\kern-.1667em\mbox{P}}}  
\newcommand{\Ex}{\mbox{{I}\kern-.1667em\mbox{E}}}
\newcommand{\Real}{\mbox{{I}\kern-.1667em\mbox{R}}}
\newcommand{\NIW}{\mathcal{NIW}}
\newcommand{\MVT}{\mathcal{MVT}}
\newcommand{\Beta}{\text{Beta}}

\newtheorem{theorem}{Theorem}

\theoremstyle{definition}

\DeclareMathOperator{\etr}{etr}

\title{Adversarial Graph Traversal}

\author{
David Banks\thanks{Department of Statistical Science, Duke University.}
\and
Elvan Ceyhan\thanks{Department of Mathematics and Statistics, Auburn University. Email: ceyhan@auburn.edu}
\and
Leah Johnson\thanks{Department of Statistical Science, Duke University.}
\and
Li Zhou\thanks{Department of Mathematics and Statistics, Auburn University.}
}

\date{}  

\begin{document}

\maketitle

\begin{abstract}
Suppose a Bayesian agent seeks to traverse a graph.
Each time she crosses an edge, she pays a price.
The first time she reaches a node, there is a payoff.
She has an opponent who can reduce the payoffs.
This paper uses adversarial risk analysis to find a solution to 
her route selection problem.
It shows how the traveler is advantaged by having an
accurate subjective distribution over the costs/payoffs
and by having a Bayesian prior for her opponent's strategic
choices.
The results are relevant to military convoy routing, corporate
competition, and certain games.
\end{abstract}


\section{Introduction}
\label{sec:intro}

Consider a traveler (female) who is moving through a 
network.
She knows her current location and the configuration of the network.
Each time she crosses an edge, she pays a cost.
The first time she reaches a node, she receives a payoff.
She does not know the costs and payoffs in advance, but she has a 
subjective distribution over those.
Her goal is to maximize her net revenue, so she terminates 
when the expected value of further travel becomes negative.

The non-adversarial version of this problem has been studied by 
\citet{caballero2025bayesian}.
It represents a general class of sequential decision problems which 
applies 
to situations such as portfolio investment over time or choosing a series 
of jobs over the course of one's career.

This paper extends that analysis to the case in which the traveler has an 
opponent (male) who seeks to minimize her reward and who is able to reduce
the payoff at a neighboring node (only once---he cannot perturb the
same node's payoff twice). 
We employ adversarial risk analysis (ARA) to let the traveler 
model her opponent's decision making \citep{rios2012adversarial}.
This problem arises in many scenarios; e.g., Prussian kriegsspiel
\citep{wintjes2021school}, or terrorists who initiate a sequence of
attacks upon infrastructure 
\citep{zhuang2010modeling, dubois2023interdicting}.
It also arises in the context of corporate competition, when two firms 
make a series of decisions over time, each trying to maximize its return
at the expense of its opponent.

The running example in this paper is convoy routing during an urban 
insurgency.
This problem was previously considered in \citet{wang2011network}.
In that paper, the ARA model was the mirroring equilibrium
solution concept \citep[cf.][section 2.2.4]{banks2015adversarial}.
In this paper, we extend that work by using the level-$k$ thinking 
solution concept \citep{stahl1995players} with a Bayesian prior over $k$.

The adversarial framework can be developed in several interesting ways.
The traveler may consider a variety of models---potentially even a 
probabilistic mixture thereof---for how her opponent thinks.
She could suppose that her opponent adopts similar models of her own
planning.
Her opponent may know the costs and payoffs, or only have subjective 
beliefs about those.
And there could be technical constraints on how her opponent can manipulate
the payoffs.
Most of these scenarios have plausible analogies in real applications.
However, in this paper, we adopt the simplest formulation 
that fully exercises the issues that arise in ARA for graph traversal.

Specifically, in our framework we assume the opponent has perfect 
knowledge of the costs and payoffs.
The traveler does not know these, but she has a subjective distribution 
over them.
And we examine two models for how the traveler conceptualizes her 
opponent's strategy:
he may greedily reduce the highest payoff among her neighboring nodes
(Type 0), or he may reduce the
second highest payoff (Type 1), since he assumes that she will anticipate 
a greedy reduction and thus choose the second-best path.
(These types correspond to the level-$k$ thinking model for her opponent, with $k$ equal to 0 or 1; see
\citet[][section 2.2.3]{banks2015adversarial} and \citet{stahl1995players} for an extended discussion.)

In the example of convoy routing, the commander's goal is to support
troops or mission goals at different intersections in a city road 
network during an insurgency.
The fog of war prevents her from knowing how much value her support can 
provide at a particular intersection, but she has a subjective
distribution on those values.
Similarly, she does not know the cost (in terms of damage from insurgents
or time/fuel/management resources) of traversing a road to reach an
intersection, but she has a subjective distribution over those, too.
Her opponent, the leader of the insurgency, knows exactly how many
combatants he has at each intersection and must guess her next 
destination in order to deploy additional forces so as to lower her
effectiveness.

Section~\ref{sec:optimization-intro} describes the optimization problem in the case without an adversary, and then extends it to cases with a Type~0 opponent or a 
Type~1 opponent.
Section~\ref{search}
uses a designed experiment to study the effects of adversary type, three different search heuristics,
and three different models for the traveler's subjective beliefs (pessimistic, accurate, and optimistic).
Section~\ref{sec:conclusion} summarizes the results.
The Appendix provides proofs and derivations.

\section{The Optimization Problem}
\label{sec:optimization-intro}

In this section, we first consider a simplified scenario in which the traveler faces no adversary.
We then show how the traveler modifies her strategy to account for an opponent.
The traveler's decisions will depend upon two theorems that are proved in the appendix.

\begin{theorem}
\label{thm:NIW-conditioning}
\emph{The Normal-Inverse-Wishart ($\NIW$) distribution is closed under conditioning.}  
\end{theorem}

\noindent
That is, if the traveler's subjective belief about the joint costs and payoffs has the $\NIW$ distribution,
then when she observes data, her new subjective belief is also $\NIW$.

\begin{theorem}
\label{thm:MVT-marginal}
\emph{Integrating out the precision matrix in the $\NIW$ distribution 
yields the multivariate $t$-distribution ($\MVT$).}  
\end{theorem}

\noindent
Both results are needed for the traveler to determine the expected value 
of her next traversal.

\subsection{Two Models}

The traveler needs two models.
One describes the payoffs and costs of travel, and the other
describes the decision-making of her opponent.

Her subjective belief about the payoffs and costs is the
$\NIW$ distribution.
As she observes values, she conditions on these, which changes
her beliefs about unobserved values.
Note that this conditioning is not the same as Bayesian updating.

Regarding her opponent, the traveler believes he is either
nonstrategic, Type 0, reducing the payoff at the adjacent node with
largest value, or strategic, Type 1, reducing the payoff at the adjacent
node that has the second-largest value.
The Type 1 opponent assumes the 
traveler expects him to be nonstrategic and so she will
avoid the node that a Type 0 player would reduce.
This framework corresponds to the opponent being a level-$k$
thinker with $k$ equal to 0 or 1 \citep{stahl1995players}.

The traveler's model for her opponent's type is Bayesian, with a 
Beta prior on him being nonstrategic.
As she observes payoffs during her traversals, she gets information
about the type of the opponent she faces, and she updates her
prior accordingly.

These two models interact because the traveler compares the 
payoff she expects to the payoff she receives, which provides
probabilistic information on whether she has a strategic or
nonstrategic opponent.
If her model for the payoffs is inaccurate, or if her prior
on the opponent is weak, or both, then she will learn slowly.
More detailed model specification is given in subsequent sections.

\subsection{Navigation in the Absence of an Adversary}
\label{sec:non-adversarial}

As the traveler navigates the graph, she observes edge costs and node payoffs.
Initially, she starts with an $\NIW$ subjective belief over all unobserved components. 
When data (i.e., the cost of a newly traversed edge or a node payoff) 
become available, she conditions her belief and obtains a new distribution on the remaining unobserved components.

The problem of finding an optimal path given uncertain costs and rewards 
is NP-hard, as shown in \citet{caballero2025bayesian}. 
Thus, even without an adversary, 
the traveler must balance the computational burden of searching 
many potential routes 
against the benefits of locating paths with high net expected values.

Our analysis assumes she uses an Uncertainty Policy: 
the extent of her search depends on how much her 
observations have changed her means for the unobserved components.
If recent observations substantially change her beliefs about
the unobserved components, then she should not evaluate 
long paths---there is too much uncertainty.
But if her beliefs about the unseen components did not change very 
much, then she can afford to consider longer paths that may lead 
to regions with large payoffs and low costs.

Specifically, suppose the traveler recalculates the conditional means 
for each unvisited edge or node after observing new data. 
Let $\tau$ be the absolute value of the maximum relative (percentage) 
change across these recalculated means, compared to her previous means. 
We adopt the following search strategy:
\begin{itemize}
\item If $\tau > 20\%$, the traveler restricts her evaluation to the 
immediate neighbors of her current node. 
\item If $10\% < \tau \leq 20\%$, the traveler expands her search to all 
two-node paths from her 
current node.
\item If $5\% < \tau \leq 10\%$, the traveler searches all three-node paths from her current node. 
\item If $\tau \leq 5\%$, the traveler is confident that the new 
information has not significantly changed 
her expectations, so she searches four-node paths from her current node.
\end{itemize}
These cutpoints were determined empirically for the road network 
used in our simulation.
For graphs with different connectivity, different cutpoints may be
better.

For some networks, the number of paths considered in this 
uncertainty policy could be so large as to be a computational
barrier.
In that case the analyst should choose some maximum number of
paths to consider and greedily sample those. 

\subsection{Navigation Against an Adversary}
\label{sec:adversarial}

We now extend the traveler's problem to account for an opponent who is 
able to reduce the payoff at one neighboring node by a fixed, known 
quantity at each turn (the same node's payoff may not 
be reduced twice). 
In our example the reduction is 30\%, or $\delta = 0.3$,
but one could use different values
for different nodes or place a subjective prior over the amounts of the 
reductions.

In the simplest model, the adversary knows the underlying true costs and 
payoffs, but the traveler only has a subjective distribution over these 
values. 
When she first navigates to a new node, she observes both the edge cost
and node payoff. 
If the observed payoff is very different from her expectation, then either 
her belief is wrong or her opponent has changed the payoff. 

The traveler believes the adversary is either Type~0, who reduces 
the largest payoff, or Type~1, who strategically reduces the second-best 
payoff, thinking the traveler expects him to penalize her best path.
If the reduced net payoff for the best node is larger than the net 
payoff for the second-best node, then the Type~1 opponent reduces the
payoff at the best node, so as to minimize her gain.

In the context of the convoy example, this structure implies that
the insurgent leader knows how many combatants he has placed at
each intersection in the road network, and he must guess where the
convoy commander will move in order to allocate his reinforcements,
which lower her effectiveness (i.e., her payoff).
The insurgent leader does not know where she will move, but he
may be a Type~0 thinker, and send reinforcements where they will
have the most value, or he may be a Type~1 thinker, and guess that the
convoy commander will anticipate a myopic choice and thus navigate
to the second-best position.
The convoy commander does not know what type of thinker her opponent
is, but she is Bayesian and she has a Beta($\alpha, \beta)$ prior on 
$\pi$, the probability that the adversary is Type~0. 

Upon her first traversal, she observes an edge cost and node payoff. 
There are four situations, based on the type of her opponent and  
possible perturbation of her node reward: 
\begin{enumerate}
    \item Her opponent is Type~0, and has reduced her payoff.  
    \item Her opponent is Type~0, and has not reduced her payoff. 
    \item Her opponent is Type~1, and has reduced her payoff. 
    \item Her opponent is Type~1, and has not reduced her payoff. 
\end{enumerate}
We shall consider each case separately.
Our example in Section~\ref{search} assumes her graph is a two-dimensional
$6 \times 6$ grid, so each node has at most four neighbors---for 
brevity and clarity, our discussion addresses that case. 
However, the reasoning generalizes immediately to any other 
(connected) network.

The traveler's $\NIW$ distribution allows her to calculate
her probability 
that each of her neighboring node payoffs would be reduced 
by a Type~0 or Type~1 adversary.
She can also calculate her probability of obtaining a node payoff 
of the size that she observed if the chosen node payoff was reduced 
and if that payoff was not reduced.
From this, she finds her posterior probability that her opponent is 
Type~0.
This probability allows her to condition on her observation to
create a new subjective distribution that is a mixture of $\NIW$
distributions.
Each time she reaches a new node, the number of components in her 
mixture doubles, creating a computational challenge in large 
networks.
One solution is to regularize by eliminating components with low 
weights.

Specifically, for each of the up-to-four neighboring nodes, 
let $\Delta_i$ for $i=1, 2, 3, 4$ be the 
corresponding payoff minus the edge cost for the $i$th neighboring 
node.
The $\MVT$ distribution allows her to calculate the probabilities 
of 24 quantities:
$\Pb(\Delta_1 > \Delta_2 > \Delta_3 > \Delta_4)$, and so forth, for 
all 24 permutations of the four indices.

Summing $\Pb(\Delta_1 > \Delta_2 > \Delta_3 > \Delta_4)$, 
$\Pb(\Delta_1 > \Delta_2 > \Delta_4 > \Delta_3)$,
$\Pb(\Delta_1 > \Delta_3 > \Delta_2 > \Delta_4)$, $\Pb(\Delta_1 > 
\Delta_3 > \Delta_4 > \Delta_2)$,
$\Pb(\Delta_1 > \Delta_4 > \Delta_2 > \Delta_3)$ and $\Pb(\Delta_1 
> \Delta_4 > \Delta_3 > \Delta_2)$
gives the probability $q_{11}$ that moving to the first node is 
the node with highest net expected value.
A slightly more involved calculation finds the probability $q_{21}$ 
that moving to the first node is the second-best choice (the best 
choice given that the true best choice has been perturbed)
and the probability $1 - q_{11} - q_{21}$ that moving to the first 
node is neither the best nor 
second-best choice.
She can also calculate these probabilities for each of the other 
up-to-three neighboring nodes, with notation $q_{1i}$ and $q_{2i}$.

When the traveler moves to node $i$, she observes an edge cost and 
a payoff.
With probability $\pi q_{1i}$ her opponent is Type~0 and has 
reduced that payoff.
With probability $(1-\pi) q_{1i}$ her opponent is Type~1 and did 
not reduce the payoff.
With probability $\pi q_{2i}$ her opponent is Type~0 and did not 
reduce the payoff.
With probability $(1-\pi) q_{2i}$ her opponent is Type~1 and did 
reduce the payoff.
The traveler then calculates her posterior probability that the 
payoff at node $i$ was reduced.

Her posterior probability that the $i$th node's payoff was 
reduced given the observed payoff $x$ is 
$$
\resizebox{1.0\columnwidth}{!}{$\displaystyle
\begin{aligned}
\Pb(\mbox{reduced } | \, x) &= \\
&\frac{f(x \, | \mbox{ reduced}) [\pi q_{1i} + (1-\pi) q_{2i}]}
{f(x \, | \mbox{ reduced}) [\pi q_{1i} + (1-\pi) q_{2i}] 
+f(x \, | \mbox{ not reduced}) [(1-\pi)q_{1i} + \pi q_{2i}]}
\end{aligned}
$}
$$
where $f(x \, | \mbox{ not reduced})$ is the marginal of the 
mixture-$\MVT$ induced by her mixture-$\NIW$ corresponding to node $i$
and $f(x \, | \mbox{ reduced})$, or 
$f(x^*)$, is the same marginal but with the observed $x$ 
rescaled to $x^* = 1/(1-\delta)$ to undo the reduction and thus 
maintain consistency with the traveler's subjective distribution.

When the traveler arrives at an unvisited node, she observes a cost 
and a payoff. 
She compares the payoff to her mixture-$\MVT$ distribution for the 
payoff. 
If the discrepancy is large, she will suspect that the opponent 
intervened. 
Depending on whether the putative manipulation aligns better with a 
Type 0 or a Type 1 opponent, 
the traveler updates her belief about $\pi$. 
Specifically, if the apparent sabotage is consistent with a 
nonstrategic change, 
the traveler increases her posterior belief that her opponent is 
Type~0; if it aligns better with the second-best reduction, 
she increases her posterior belief on Type~1.

Let $\gamma$ denote the traveler’s updated belief that the observed 
discrepancy results from adversarial reduction of the node payoff.
The calculation is straightforward but entails consideration of 
several cases, and so 
the derivation is relegated to the Appendix, which shows
\begin{equation} 
\label{gamma}
\resizebox{1.0\columnwidth}{!}{$\displaystyle
\gamma
= \frac{\left[f(x^*)q_{1i} + f(x)(1 - q_{1i})\right] \pi}{f(x^*)
(\pi q_{1i}
+ (1 - \pi)q_{2i}) 
+ f(x)(\pi (1 - q_{1i}) + (1- \pi)(1 - q_{2i}))}.$}
\end{equation}
Upon observing the payoff, she updates her Beta prior from
$\Beta(\alpha,\beta)$ to
\[
  \Beta\bigl(\alpha + \gamma,\; \beta + 1 - \gamma\bigr).
\]

At this point, the traveler's decision is straightforward.
In order to maximize her expected utility for a one-step search,
she calculates the expected value of traveling to each neighboring 
node, where the expected value incorporates the possibility that
the opponent has reduced the payoff.
Similarly, for longer search paths, she chooses the next 
step so as to maximize her expected net value over that path.

\subsection{Fixed-Path vs.\ Adaptive Choice}
We have described an adaptive strategy for the traveler in which she 
learns as she observes costs and payoffs.
Alternatively, one might consider a fixed strategy, in which she 
chooses her entire path at the outset.
Under a mild condition, we show that the adaptive strategy 
is at least as good as the fixed strategy.
The proof is in the Appendix.

\begin{theorem}
\label{thm:expectations}
The traveler may choose a fixed path $\eta_{\mathrm{F}}$ 
at the outset, or adaptively select a path $\eta_{\mathrm{A}}$ 
based on observations made over time. 
Let $R_{\mathrm{F}}$ and $R_{\mathrm{A}}$ denote the net rewards
under each strategy. 
For any subjective distribution over the node rewards and
edge costs,
\[
\mathbb{E}[R_{\mathrm{A}}] \ge \mathbb{E}[R_{\mathrm{F}}].
\]
\end{theorem}

Note that this result is true for any distribution, not just
the $\NIW$ family. 
This generality is important because the adaptive traveler
has a subjective distribution that is a mixture of 
$t$-distributions.

\section{Search and Simulation}
\label{search}

\begin{figure*}[!t]
  \centering
  \begin{minipage}[b]{0.28\textwidth}
    \includegraphics[width=\textwidth]{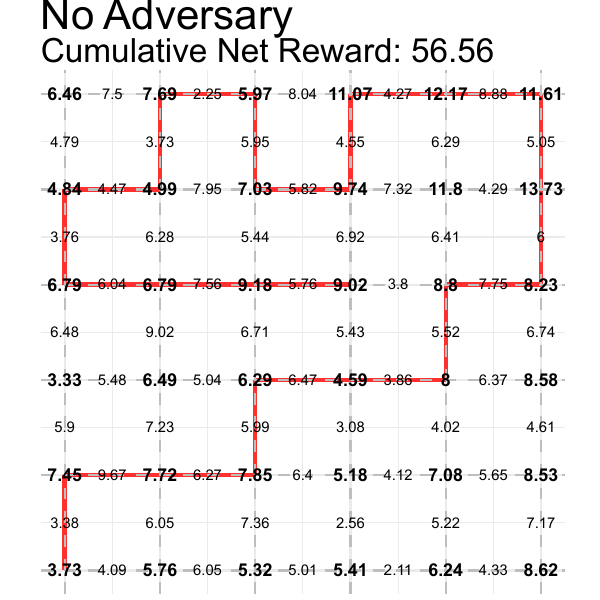}
  \end{minipage}
  \hfill
  \begin{minipage}[b]{0.28\textwidth}
    \includegraphics[width=\textwidth]{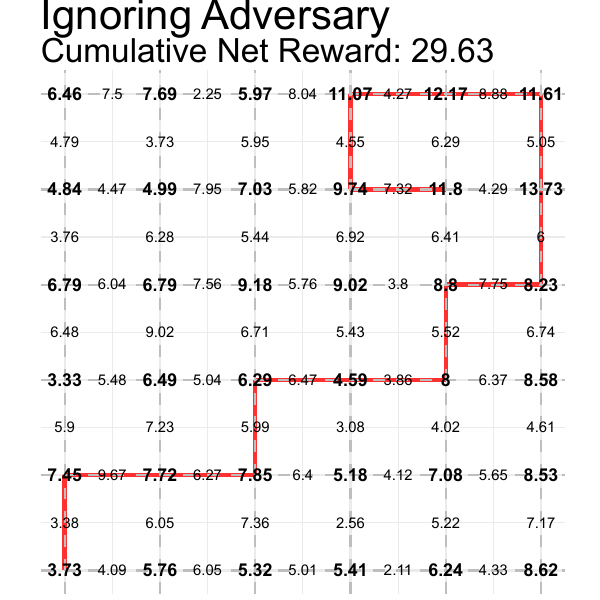}
  \end{minipage}
  \hfill
  \begin{minipage}[b]{0.28\textwidth}
    \includegraphics[width=\textwidth]{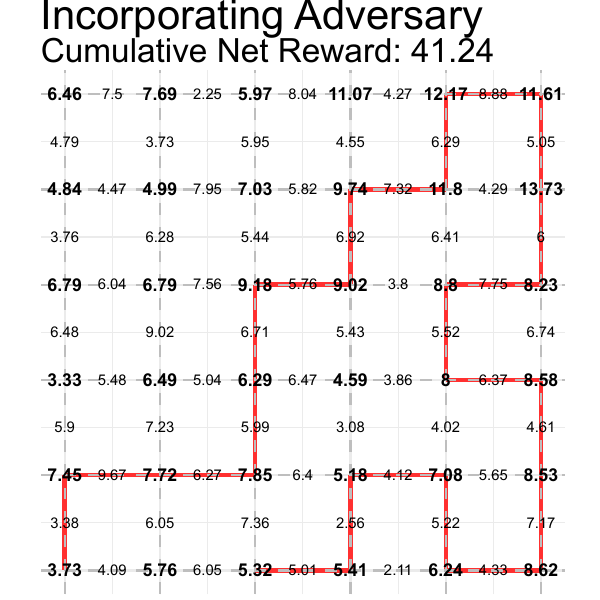}
  \end{minipage}
  \caption{\normalsize An illustration of the three situations (No Adversary, Ignoring
  Adversary, and Incorporating Adversary) in our 6-by-6 grid simulation.}
  \label{fig:grid}
\end{figure*}

\citet{caballero2025bayesian} does a simple simulation to compare four search 
strategies for network traversal in the absence of an adversary:  
myopic (greedy, or one-step look-ahead), an upper confidence 
bound policy, the $H$-path policy \citep{andreas2008branch}, 
and the speculating clairvoyant policy. 
It finds that the last two are the best and perform about equally well.

Our simulation adopts a factorial analysis of variance (ANOVA) to
study main effects and interactions between four factors:
search strategy, adversary type, the accuracy of the traveler's
subjective distribution, and the traveler's knowledge of
whether or not she has an opponent.
The response is the average net reward.
Using ANOVA to study factorial effects in a simulation is common
practice \citep[cf.][]{banks1988histospline,banks2003comparing},
often after making some transformation to normalize the response.
In this case, the response is an average so the Central Limit
Theorem obviates the need for transformation.

The search strategy factor has three levels---the benchmark myopic 
search, the $H$-path policy with $H=3$, as in \citet{caballero2025bayesian}, in which
one searches all possible paths of length 3, and our proposed
Uncertainty Policy search from Section~2.2.
The adversary type factor also has three levels---no adversary, the
Type 0 opponent, and the Type 1 opponent.
For the accuracy factor, we allow the traveler to be pessimistic,
and think that the mean payoffs are 10\% less than the truth,
or optimistic, and think that payoffs are 10\% greater than the 
truth, or to be accurate about the means.
The last factor is whether or not the traveler thinks she has
an opponent.

Our simulations use a 6-by-6 grid, which mimics the urban road 
network on which the convoy operates.
The nodes are associated with the Cartesian lattice
coordinates from (1, 1) to (6, 6), so each interior node has four neighbors
and the network has 36 nodes and 60 edges.
All payoffs and costs in the simulation have the multivariate Gaussian 
distribution.
The expected payoff for the node with coordinates $(x, y)$ is $x+y$, so
payoffs increase as one moves to the northeast.
The expected cost for an edge connecting nodes $(x_1, y_1)$ to $(x_2, y_2)$
is $(x_1 + y_1 + x_2 + y_2)/4$, so expected cost increases with the 
expected payoff.
The covariance matrix $\Sigma$ reflects spatial correlation decaying with
the length of the shortest path connecting the nodes or edges.
Specifically, let $\rho = 0.5$ be the base correlation parameter and 
define $d_{ij} = |x_i - x_j| + |y_i - y_j|$.
Then the covariance between two nodes at $(x_1, y_1)$ and $(x_2, y_2)$ 
is $\rho^{\,2d_{12}-1}$.
For two edges, one connecting $(x_1, y_1)$ to $(x_2, y_2)$ and the 
other connecting $(x_3, y_3)$ to $(x_4, y_4)$, 
their covariance equals $\rho$ if they share a node,
otherwise it is $\rho^{2d+1}$, where 
$d = \min\{d_{13},d_{14},d_{23},d_{24}\}$.
The covariance between a node at $(x_1, y_1)$ and an edge between 
$(x_2, y_2)$ and $(x_3, y_3)$ is
$\rho$ if the node is one endpoint of the edge, 
and $\rho^{2d}$ otherwise, where $d =\min\{d_{12},d_{13}\}$.

In our simulations, the traveler believes the payoffs and costs
follow a 
$\NIW_p(\vect{\mu}_0, \vect{\Sigma}_0, \vect{\Psi}_0, \nu_0)$
distribution.
We consider three cases: (1) an accurate prior with mean equal to 
the true expected values, (2) a pessimistic prior that
underestimates all expected values by 10\%, and (3) an optimistic
prior that overestimates all means by 10\%.
The prior covariance in the pessimistic, optimistic and accurate
cases equals the $\vect{\Sigma}_0$ in the generative model.

For all cases, the degrees of freedom is $\nu_0 = p+4$, which 
exceeds the minimum proper threshold $p+1$, ensuring a weakly
informative subjective distribution that is proper and which allows
flexible adaptation as costs and rewards are observed.
The $\NIW$ scale matrix $\vect{\Psi}_0$ is chosen so that
$\vect{\Sigma}_0 = \vect{\Psi}_0/(\nu_0 - p -1)$, making 
$\vect{\Sigma}_0$ the expectation of the inverse Wishart distribution.
The traveler's prior on the probability that her opponent is Type 0 
is Beta(1, 1).

We generate 100 networks, with costs and payoffs, as draws from
$N(\vect{\mu}, \vect{\Sigma})$, where $\vect{\mu}$ and $\vect{\Sigma}$
are realizations from the specified prior.
To reduce variance in comparisons, the same 100 networks are used
for all levels of the factors in the ANOVA.
For each network we have a Type 0 and Type 1 opponent, we
have the accurate, pessimistic, and optimistic traveler, and
we allow the traveler to account for, or not account for, her
opponent.
We also use the three search strategies for each of the 100 networks,
and compare the average cumulative rewards when the search terminates 
in an ANOVA.

Figure~\ref{fig:grid} shows three situations, which illustrate the 
simulation results.
The left-hand path shows the case in which there is no adversary, and the 
traveler myopically (one-step look-ahead) chooses the next step.
The middle path is the case in which a Type 0 adversary is present, but 
the traveler does not account for him and uses myopic search.
The right-hand path represents the case in which the traveler has a 
uniform prior on the adversary being Type 0, and she makes her myopic 
best choice.
When there is no adversary, for this realization of costs and payoffs, the 
cumulative reward is 56.56.
Believing that there is no adversary gives a bad outcome---the reward 
is 29.63.
That is substantially improved by taking account of the adversary, in 
which case the reward is 41.24.

\begin{table*}[!ht]
    \centering
    \caption{ANOVA summary for the fully crossed model predicting traveler net 
    reward. Main effects and higher‐order interactions among policy, 
    belief, adversary type, and strategy are shown with degrees of freedom 
    (df), sum of squares, mean squares, $F$‐statistics, and $p$‐values. 
    Significant terms $(p < .05)$ are in \textbf{bold}.\label{table:anova}}
    {
    \begin{tabular}{lrrrrr}
    \toprule
    Term & {df} & {Sum Sq} & {Mean Sq} & {F} & {p} \\ 
     \midrule
    \textbf{policy} & 2.00 & 44333.57 & 22166.79 & 170.90 & 0.00 \\ 
    \textbf{belief} & 2.00 & 1629.39 & 814.70 & 6.28 & 0.00 \\ 
    \textbf{adversary} & 2.00 & 1476595.75 & 738297.88 & 5692.18 & 0.00 \\ 
    \textbf{strategy} & 1.00 & 830.59 & 830.59 & 6.40 & 0.01 \\ 
    \midrule 
    policy:belief & 4.00 & 34.99 & 8.75 & 0.07 & 0.99 \\ 
    \textbf{policy:adversary} & 4.00 & 10049.38 & 2512.35 & 19.37 & 0.00 
    \\ 
    \textbf{belief:adversary} & 4.00 & 2079.67 & 519.92 & 4.01 & 0.00 \\ 
    policy:strategy & 2.00 & 21.73 & 10.87 & 0.08 & 0.92 \\ 
    belief:strategy & 2.00 & 208.88 & 104.44 & 0.81 & 0.45 \\ 
    \textbf{adversary:strategy} & 2.00 & 18512.49 & 9256.24 & 71.36 & 0.00 
    \\ 
    \midrule
    policy:belief:adversary & 8.00 &  207.63 & 25.95 & 0.20 & 0.99 \\ 
    policy:belief:strategy & 4.00 & 241.88 & 60.47 & 0.47 & 0.76 \\ 
    policy:adversary:strategy & 4.00 & 1061.57 & 265.39 & 2.05 & 0.09 \\ 
    belief:adversary:strategy & 4.00 & 157.10 & 39.27 & 0.30 & 0.88 \\ 
    \midrule
    policy:belief:adversary:strategy & 8.00 & 478.42 & 59.80 & 0.46 & 0.88 
    \\ 
    \midrule
    Residuals & 5346.00 & 693396.52 & 129.70 &  &  \\ 
    \bottomrule
    \end{tabular}
    }
\end{table*}

\begin{figure*}[!ht]
    \centering
    \includegraphics[width=0.85\linewidth]{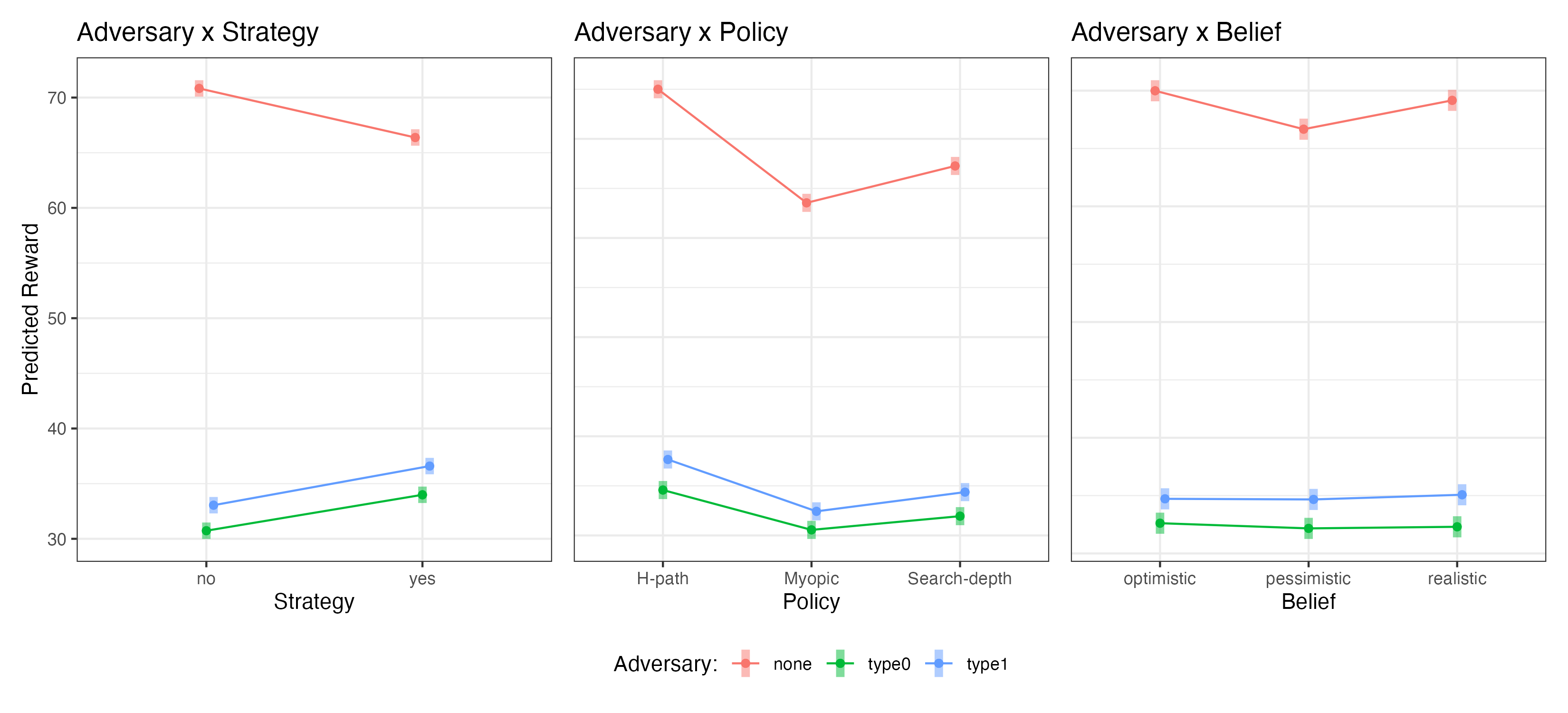}
    \caption{Three significant interaction plots from the fully crossed 
    ANOVA model showing how adversary type interacts with decision 
    strategy, traversal policy, and belief model in predicting average net 
    reward.}
    \label{fig:anova}
\end{figure*}

\begin{figure}[!ht]
    \centering
    \includegraphics[width=0.8\linewidth]{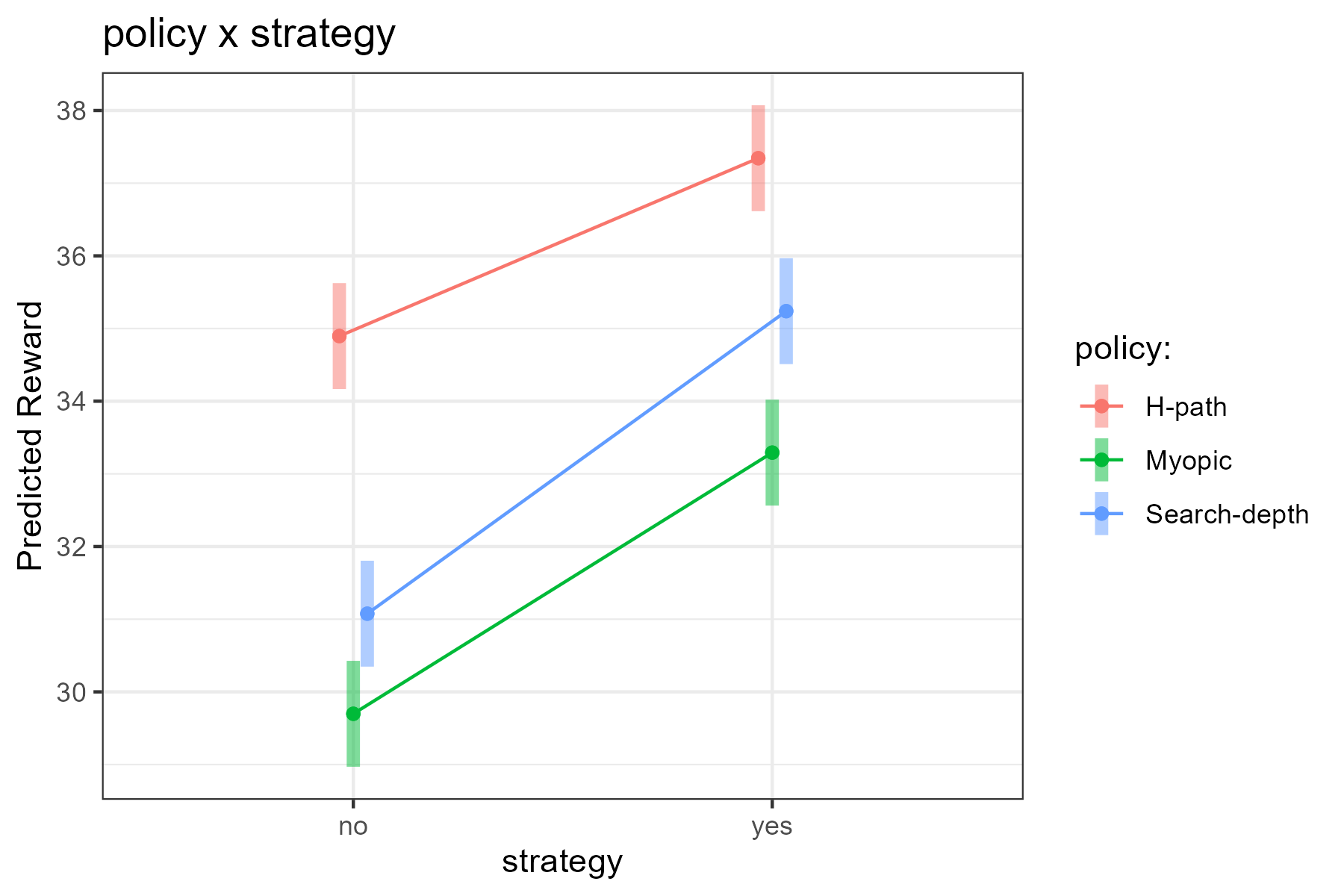}
    \caption{Significant policy-strategy interaction from the ANOVA model 
    using adversary-present runs, illustrating the different reward gains 
    for the Uncertainty Policy heuristic, the $H$-path, and the myopic policy.}
    \label{fig:anova_subset}
\end{figure}

\begin{figure}[!ht]
    \centering
    \includegraphics[width=0.8\linewidth]{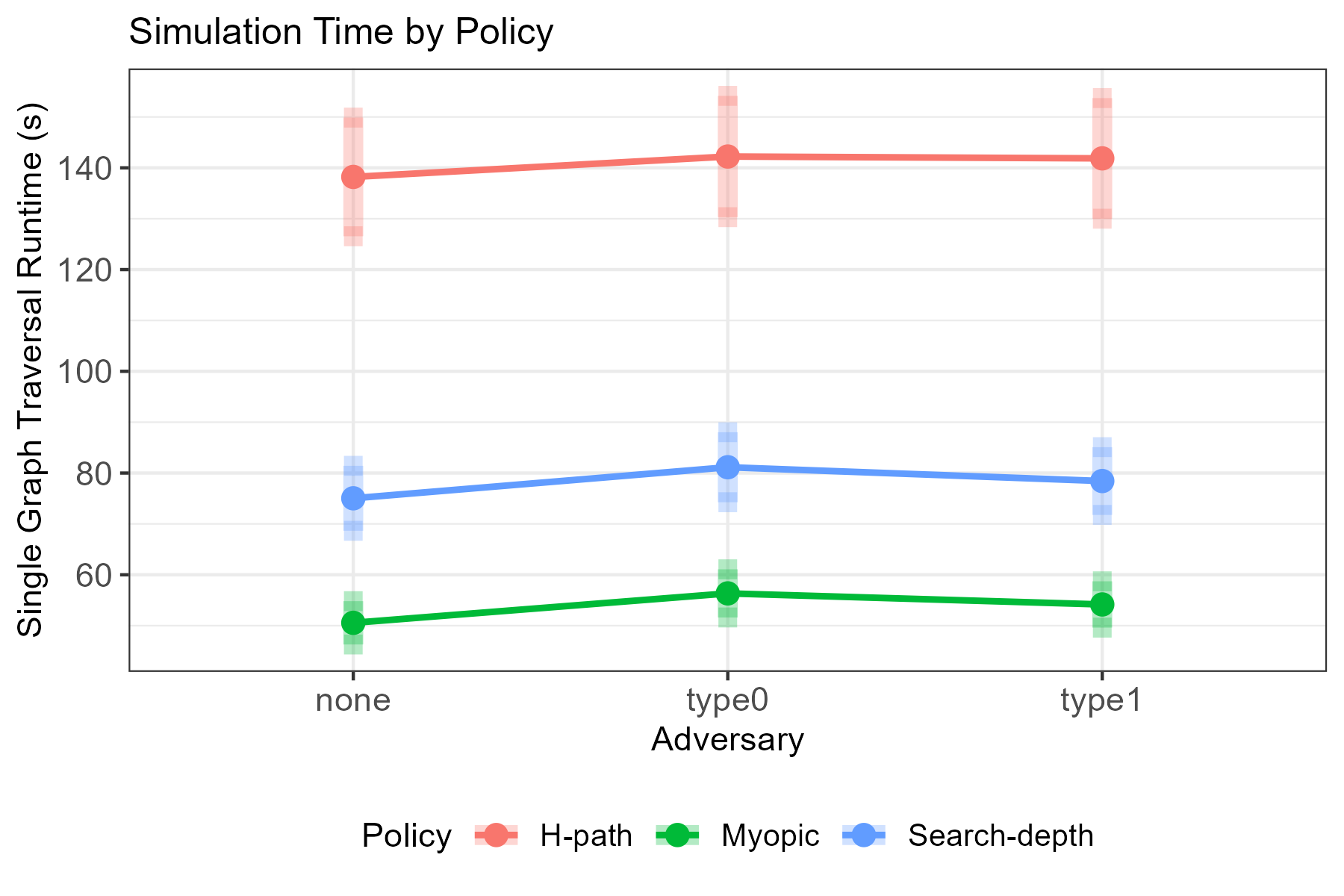}
    \caption{Simulation time comparison across adversarial type, 
    illustrating computational efficiency differences between the Uncertainty Policy, the $H$-path, and the myopic policy.}
    \label{fig:runtime}
\end{figure}

Table~\ref{table:anova} summarizes the four-factor ANOVA, with 
significant main effects and interactions in bold. 
Figure~\ref{fig:anova} explores the significant two-way interactions, 
showing the traveler's predicted net rewards under our model 
broken out by decision strategy, traversal policy, and belief, 
for both types of adversary. 

Adversarial presence has a strong main effect: in the absence of an
adversary, the expected rewards are about 70, but collapse to 
the 30-40 range when an adversary is present. 
The predicted reward is higher under a Type 1 adversary than a Type 0
adversary in all cases, which is to be expected since 
the Type 1 adversary penalizes the second-best rather than the best move,
so rewards are slightly higher. 

The first panel shows the impact of the traveler's knowledge of the
adversary. 
When no adversary is present, falsely assuming one is associated 
with nearly a 5-point loss in average predicted reward, while 
accounting for an adversary when one is present is associated with a
comparable gain. 
Travelers with correct knowledge tend to outscore those 
that are misinformed.   

The second panel compares traversal policies under various types 
of opponent. 
While all policies suffer under adversarial interference, the 
performance gap between our Uncertainty Policy heuristic and the $H$-path 
policy narrows. 
In all cases, the Uncertainty Policy outperforms the myopic baseline. 
Our approach reduces the computational burden associated with the 
$H$-path policy, as seen in Figure \ref{fig:runtime} which compares the
time taken to complete a single run under each policy, across 100
replicates.

The third panel shows the impact of the traveler's beliefs under 
different adversary types. 
In the presence of an adversary, there is virtually no difference 
in the average predicted rewards across beliefs.
However, in the absence of an adversary, pessimistic travelers 
appear to perform worse than optimistic or realistic travelers.

Taken together, these panels show that adversary presence has a strong 
association with expected net reward, and our proposed policy delivers
robust performance across adversarial types, narrowing the gap with the
$H$-Path policy while dramatically reducing computational cost and
outperforming the myopic policy in all cases.

We fit a supplementary ANOVA restricted to cases in which an adversary 
is present, and the significant policy-strategy interaction is shown in 
Fig.~\ref{fig:anova_subset}. 
Incorporating belief-updating boosts predicted rewards for every 
policy, but the increase is most pronounced for the Uncertainty Policy,
indicating that our proposed policy pairs especially well with 
adversary learning.

\section{Conclusion}\label{sec:conclusion}

This paper studies an agent's sequential decision making when there is an
opponent who can lower the expected payoff from the agent's choice.
This is a broad class of problems, and versions arise in corporate 
competition and certain strategic games.
It is helpfully concrete to represent this sequential decision problem
under uncertainty as traversing a graph with unknown payoffs and costs,
and the running example in this paper is convoy routing during an
urban insurgency.

There are many ways this general problem framework could be built out.
We use the simplest version that fully illustrates the dynamics of the
interplay between the agent's subjective distribution, and the type of 
adversary, and the optimization policy.
But one could have more complex situations in which the opponent
does not have perfect knowledge of the payoffs and costs, holds a 
different subject distribution than the traveler, and must learn
by conditioning as she does.
That case opens the door to the mirror equilibrium adversarial risk
analysis \citep{banks2015adversarial}, instead of the level-$k$
thinking school of \citet{stahl1995players}.

Our ANOVA results show the relative magnitude of four different
factors in our designed experiment, emphasizing the advantage
the traveler receives from using adversarial risk analysis.
Her prior on the type of her opponent leads to a novel
mixture model of multivariate $t$-distributions for her posterior.
And our results significantly extend the work of 
\citet{caballero2025bayesian} and \citet{wang2011network}.

\section*{Appendix}
The traveler needs two theorems to decide upon her next traversal. 
Both are quite obvious, but we could find no proof of the first in
easily available literature, and the second is stated but not 
proved in \citet{bishop2006pattern}.
For completeness, we include the derivations.

\vspace{.25cm}

\noindent
Theorem~\ref{thm:NIW-conditioning}.
The Normal-Inverse-Wishart distribution is closed under conditioning.

\noindent
\emph{Proof.}  
Let $\mathbf{X}$ be a $p$-dimensional random vector with Normal-Inverse-Wishart distribution 
$\NIW_p(\boldsymbol{\mu}, \boldsymbol{\Lambda}, \boldsymbol{\Psi}, \nu)$, 
where $\boldsymbol{\mu}$ is the mean vector, $\boldsymbol{\Lambda}$ is the precision matrix,
$\boldsymbol{\Psi}$ is the scale matrix and $\nu$ is the degrees of freedom.
Let $\boldsymbol{\Sigma}$ be its associated covariance matrix (with precision $\boldsymbol{\Lambda} = \boldsymbol{\Sigma}^{-1}$).  
By definition, the joint density of $\mathbf{X}$ factors as:
$$
f(\mathbf{x}|\boldsymbol{\mu}, \boldsymbol{\Lambda}, \boldsymbol{\Psi}, \nu) = \mathcal{N}_p(\mathbf{x}
|\boldsymbol{\mu}, \boldsymbol{\Lambda}) \mathcal{IW}_p(\boldsymbol{\Lambda}| \boldsymbol{\Psi}, \nu)
$$
where $\mathcal{N}_p(\mathbf{x} |\boldsymbol{\mu}, \boldsymbol{\Lambda})$ is a multivariate normal density 
and $\mathcal{IW}_p(\boldsymbol{\Lambda}| \boldsymbol{\Psi}, \nu)$ is an inverse-Wishart density.

We want to find the conditional density of $\mathbf{X}_{k+1:p}$ given $\mathbf{X}_{1:k}$. 
Partition $\mathbf{X}$ into $\mathbf{X}_{1:k}$ and $\mathbf{X}_{k+1:p}$,
and similarly partition $\boldsymbol{\mu}$, $\boldsymbol{\Lambda}$, and $\boldsymbol{\Psi}$.  
So, the joint density of $\mathbf{X}$ can be decomposed as
$f(\mathbf{x}) = f(\mathbf{x}_{1:k}, \mathbf{x}_{k+1:p})$.
Partitioning the mean vector and precision matrix and covariance yields:
$$
\boldsymbol{\mu} = \begin{pmatrix}
\boldsymbol{\mu}_1 \\
\boldsymbol{\mu}_2
\end{pmatrix}, \quad \boldsymbol{\Lambda} = \begin{pmatrix}
\boldsymbol{\Lambda}_{11} & \boldsymbol{\Lambda}_{12} \\
\boldsymbol{\Lambda}_{21} & \boldsymbol{\Lambda}_{22}
\end{pmatrix}
$$
where $\boldsymbol{\mu}_1$ and $\boldsymbol{\Lambda}_{11}$ correspond to the first $k$ components 
and $\boldsymbol{\mu}_2$ and $\boldsymbol{\Lambda}_{22}$ correspond to the remaining $p-k$
components.
Since $\boldsymbol{\Lambda}$ is symmetric (it is the inverse of a covariance matrix), the off–diagonal blocks satisfy
$
\boldsymbol{\Lambda}_{21} = \boldsymbol{\Lambda}_{12}^\top.
$
Here $\boldsymbol{\Lambda}_{12} \in \mathbb{R}^{k\times (p-k)}$ collects the cross–precision terms 
that couple the first $k$ variables with the remaining $p-k$ variables.
The covariance matrix $\boldsymbol{\Sigma}$ is similarly partitioned.
Standard multivariate normal theory implies that
the conditional distribution of $\mathbf{x}_{k+1:p}$ 
given $\mathbf{x}_{1:k}$ is
\[
  \mathbf{X}_{k+1:p} \,\big\vert\, \mathbf{X}_{1:k},\, \boldsymbol{\Lambda} 
  \;\sim\; \mathcal{N}_{p-k}\!\bigl(\boldsymbol{\mu}_{2\mid1},\, \boldsymbol{\Sigma}_{22|1} \bigr),
\]
where
$
\boldsymbol{\mu}_{2|1} = \boldsymbol{\mu}_2 + \boldsymbol{\Lambda}_{21} \boldsymbol{\Lambda}_{11}^{-1} 
(\mathbf{x}_{1:k} - \boldsymbol{\mu}_1)$
 and 
$\boldsymbol{\Sigma}_{22|1} = \boldsymbol{\Lambda}_{22}^{-1} - \boldsymbol{\Lambda}_{21} 
\boldsymbol{\Lambda}_{11}^{-1} \boldsymbol{\Sigma}_{11|1} \boldsymbol{\Lambda}_{11}^{-1} \boldsymbol{\Lambda}_{12}.
$
Using the Schur complements \citep[see, e.g.,][]{mardia1979multivariate} gives
$$
\boldsymbol{\Sigma}_{22|1} = \boldsymbol{\Lambda}_{22}^{-1} - \boldsymbol{\Lambda}_{21} 
\boldsymbol{\Lambda}_{11}^{-1} \boldsymbol{\Lambda}_{12}.
$$

Next, we note that $\boldsymbol{\Lambda}$ itself is distributed as $\mathcal{IW}_p(\boldsymbol{\Psi}, \nu)$.  
By partitioning $\boldsymbol{\Lambda}$ and $\boldsymbol{\Psi}$ conformably, one can show that
\[
  \boldsymbol{\Lambda}_{22\mid1}
  \sim \mathcal{IW}_{p-k}\!\bigl(\boldsymbol{\Psi}_{22\mid1},\, \nu - k\bigr),
\]
where $\boldsymbol{\Psi}_{22|1} = \boldsymbol{\Psi}_{22} - \boldsymbol{\Psi}_{21} \boldsymbol{\Psi}_{11}^{-1} \boldsymbol{\Psi}_{12}$
and $\nu - k$ are the reduced dimensional Inverse-Wishart parameters after conditioning on $\mathbf{X}_{1:k}$.  
Combining these facts reveals that $\bigl(\mathbf{x}_{k+1:p} | \mathbf{x}_{1:k} , \boldsymbol{\Lambda}_{22\mid1}\bigr)$ 
follows an $\NIW_{p-k}$ distribution with suitably updated parameters
$\bigl(\boldsymbol{\mu}_{2\mid1},\, \boldsymbol{\Lambda}_{22\mid1},\, \boldsymbol{\Psi}_{22\mid1},\, \nu - k\bigr)$.  
Hence,
\[
\bigl(\mathbf{x}_{k+1:p} | \mathbf{x}_{1:k} , \boldsymbol{\Lambda}_{22\mid1}\bigr)
  \sim
  \NIW_{p-k}\bigl(\boldsymbol{\mu}_{2\mid1},\, \boldsymbol{\Lambda}_{22\mid1},\, \boldsymbol{\Psi}_{22\mid1},\, \nu - k\bigr).
  \;\ \square
\]

\medskip

\noindent
Theorem 2. Integrating out the precision matrix in the $\NIW$ 
distribution gives the multivariate $t$-distribution.

\noindent
\emph{Proof:}  Let $X$ have $\NIW$ distribution, so
\[
\begin{aligned}
    \mathbf{x} \mid \boldsymbol{\mu}, \boldsymbol{\Sigma} &\sim \mathcal{N}(\boldsymbol{\mu}, \boldsymbol{\Sigma}), \\
    \boldsymbol{\Sigma} \mid \boldsymbol{\Psi}, \nu &\sim \mathcal{W}^{-1}(\boldsymbol{\Psi}, \nu),
\end{aligned}
\]
where $\boldsymbol{\Psi}$ is the Inverse-Wishart scale matrix, and $\nu$ is the degrees of freedom.
Let $\mbox{etr}(\mathbf{A}) = \exp[\mbox{tr}(\mathbf{A})]$.
Then integrating out $\boldsymbol{\Sigma}$ gives the multivariate $t$-distribution ($\MVT$) since
\[\resizebox{1.0\columnwidth}{!}{$\displaystyle
\begin{aligned}
    f\left(\mathbf{x} \mid \boldsymbol{\mu}, \boldsymbol{\Psi}, \nu\right)
    &= \int f\left(\mathbf{x} \mid \boldsymbol{\mu}, \boldsymbol{\Sigma}\right)
         \, f\left(\boldsymbol{\Sigma} \mid \boldsymbol{\Psi}, \nu\right) \, d\boldsymbol{\Sigma} \\
    &\propto \int \left|\boldsymbol{\Sigma}\right|^{-1/2}\exp\left\{-\frac{1}{2}(\mathbf{x} - \boldsymbol{\mu})^{\top}\boldsymbol{\Sigma}^{-1} (\mathbf{x} - \boldsymbol{\mu})\right\} \\
    &\quad \times \left|\boldsymbol{\Sigma}\right|^{-(\nu + p + 1)/2}\exp\left\{-\frac{1}{2}\mbox{tr}\left(\boldsymbol{\Psi}\boldsymbol{\Sigma}^{-1}\right)\right\} d\boldsymbol{\Sigma} \\
    &= \int \left|\boldsymbol{\Sigma}\right|^{-(\nu + p + 2)/2}\exp\left\{-\frac{1}{2}\mbox{tr}\Big[(\mathbf{x} - \boldsymbol{\mu})(\mathbf{x} - \boldsymbol{\mu})^{\top}\boldsymbol{\Sigma}^{-1}\Big]\right\} \\
    &\quad \times \exp\left\{-\frac{1}{2}\mbox{tr}\left(\boldsymbol{\Psi}\boldsymbol{\Sigma}^{-1}\right)\right\} d\boldsymbol{\Sigma} \\
    &= \int \left|\boldsymbol{\Sigma}\right|^{-(\nu + p + 2)/2} \etr\left\{-\frac{1}{2}\Big[(\mathbf{x} - \boldsymbol{\mu})(\mathbf{x} - \boldsymbol{\mu})^{\top} + \boldsymbol{\Psi}\Big]\boldsymbol{\Sigma}^{-1}\right\} d\boldsymbol{\Sigma}.
\end{aligned}$}
\]

Making the change of variable $\mathbf{S} = \boldsymbol{\Sigma}^{-1}$, so that
\[
    d\boldsymbol{\Sigma} = \left|\mathbf{S}\right|^{-(p+1)}\, d\mathbf{S},
\]
yields
\[\resizebox{1.0\columnwidth}{!}{$\displaystyle
\begin{aligned}
    f\left(\mathbf{x} \mid \boldsymbol{\mu}, \boldsymbol{\Psi}, \nu\right) 
    &\propto \int \left|\mathbf{S}\right|^{(\nu + p + 2)/2} \mbox{etr}\left\{-\frac{1}{2}\Big[(\mathbf{x} - \boldsymbol{\mu})(\mathbf{x} - \boldsymbol{\mu})^{\top} + \boldsymbol{\Psi}\Big]\mathbf{S}\right\} \\
    &\quad \times \left|\mathbf{S}\right|^{-(p + 1)} d\mathbf{S} \\
    &= \int \left|\mathbf{S}\right|^{(\nu - p)/2} \mbox{etr}\left\{-\frac{1}{2}\Big[(\mathbf{x} - \boldsymbol{\mu})(\mathbf{x} - \boldsymbol{\mu})^{\top} + \boldsymbol{\Psi}\Big]\mathbf{S}\right\} d\mathbf{S}.
\end{aligned}$}
\]

The integrand is a Wishart kernel in $\mathbf{S}$, with degrees of freedom $\nu + 1$ and scale matrix 
$\Big[(\mathbf{x} - \boldsymbol{\mu})(\mathbf{x} - \boldsymbol{\mu})^{\top} + \boldsymbol{\Psi}\Big]^{-1}$. 
The integral is proportional to the inverse of the corresponding normalizing constant, so that 
\[
\begin{aligned}
    f\left(\mathbf{x} \mid \boldsymbol{\mu}, \boldsymbol{\Psi}, \nu\right) 
    &\propto \left|(\mathbf{x} - \boldsymbol{\mu})(\mathbf{x} - \boldsymbol{\mu})^{\top} + \boldsymbol{\Psi}\right|^{-(\nu + 1)/2}.
\end{aligned}
\]
Applying the Matrix-Determinant Lemma \citep{harville1998matrix} yields
\[\resizebox{1.0\columnwidth}{!}{$\displaystyle
\begin{aligned}
    f\left(\mathbf{x} \mid \boldsymbol{\mu}, \boldsymbol{\Psi}, \nu\right) 
    &\propto \left|(\mathbf{x} - \boldsymbol{\mu})^{\top}\boldsymbol{\Psi}^{-1}(\mathbf{x} - \boldsymbol{\mu}) + 1\right|^{-(\nu + 1)/2} \\
    &= \left|\frac{1}{\nu + 1 - p}(\mathbf{x} - \boldsymbol{\mu})^{\top}\Big(\frac{\boldsymbol{\Psi}}{\nu + 1 - p}\Big)^{-1}(\mathbf{x} - \boldsymbol{\mu}) + 1\right|^{-(\nu + 1 - p + p)/2}.
\end{aligned}$}
\]

By making the change of variables,
\[
    \nu' = \nu - p + 1 \quad \text{and} \quad \boldsymbol{\Psi}' = \frac{\boldsymbol{\Psi}}{\nu-p+1},
\]
we get
\[
\begin{aligned}
    f\left(\mathbf{x} \mid \boldsymbol{\mu}, \boldsymbol{\Psi}, \nu\right) 
    &\propto \left|\frac{1}{\nu'}(\mathbf{x} - \boldsymbol{\mu})^{\top}{\boldsymbol{\Psi}'}^{-1}(\mathbf{x} - \boldsymbol{\mu}) + 1\right|^{-(\nu' + p)/2},
\end{aligned}
\]
which is the kernel of a $t_{\nu'}\left(\boldsymbol{\mu}, \boldsymbol{\Psi}'\right)$ distribution.
\hspace{.5in} $\square$

\medskip

\vspace{.25cm}

\noindent
Theorem~\ref{thm:expectations}.
Under mild conditions, an adaptive traveler has higher expected
net reward than a traveler who chooses a fixed path at the outset.

\noindent
\emph{Proof.}  
Let $\boldsymbol{r} = (r_1, \ldots, r_N)^\top$ denote  
vector of random node rewards, and let 
$\boldsymbol{c} = (c_1, \ldots, c_M)^\top$ be the vector of 
random edge costs.
The net reward along a path $\eta$ is:
\[
R(\eta) = \sum_{v \in \eta} r_v - \sum_{e \in \eta} c_e,
\]
where $v$ indexes the nodes in $\eta$ and $e$ the edges in $\eta$.

We assume that the traveler either does exhaustive search
of the possible paths or uses incremental heuristic search
\citep{hart1968formal}.
(This mild assumption ensures that the traveler is not using a foolish
policy which does not attempt to increase her net reward.)

Let $\eta_1$ denote the first traversal, and let $X_1$ be
the first net reward.
After reaching the first node, the adaptive traveler chooses the
next traversal $\eta_2(X_1)$ so as to maximize her
expected net reward. 
And the fixed-path traveler proceeds to $\eta_2$ as planned. 

Let $X_2 = R(\eta_2(X_1))$ be the net reward for the adaptive 
continuation, and let $R_{\mathrm{F}} = X_1 + R(\eta_2^0)$.
Then
\[
\mathbb{E}[X_2 \mid X_1] \ge \mathbb{E}[R(\eta_2)],
\]
since the traveler uses a sensible (incremental heuristic)
policy to choose $\eta_2(X_1)$.  
The law of iterated expectations implies:
\[
\mathbb{E}[R_{\mathrm{A2}}] = \mathbb{E}[X_1 + \mathbb{E}[X_2 \mid X_1]] 
\ge \mathbb{E}[X_1 + \mathbb{E}[R(\eta_2)]] = \mathbb{E}[R_{\mathrm{F2}}],
\]
where $R_{\mathrm{A2}}$ is the adaptive-path reward after two
steps and $R_{\mathrm{F2}}$ is the fixed-path reward after two
steps.
Repeating this argument for all next steps completes the proof.
Note that the adaptive policy may terminate before or after 
the fixed-path policy. 
\hspace{2.85 in} $\square$

\medskip

\subsection*{Calculating the Probability $\gamma$}

The traveler is located at a node in a grid, and considers moving to 
one of up to four neighboring nodes. 
For each neighbor edge-node pair $i \in \{1, 2, 3, 4\}$, define 
(1) $\Delta_i \coloneq i^{th}\text{ node reward} - \text{edge cost}$, or
the net gain in moving to node $i$, and (2)
$\Delta_i^{(\delta)} \coloneq (1 - \delta) 
(i^{th} \text{ node reward}) - \text{edge cost}$, or the net gain in 
moving to node $i$ when the adversary reduced the node reward 
by a fixed percentage $\delta$ (in our case, 30\%).  

The traveler has a subjective mixture-$\MVT$ distribution over 
the true costs and
rewards, which is inherited by the linear transformation 
$\vect{\Delta} = (\Delta_1, \Delta_2, \Delta_3, \Delta_4)$. 
Denote the density of $\vect{\Delta}$ by $f$. 
The traveler can now compute
$$
\scalebox{1.0}{$\displaystyle
\begin{array}{rcl}
q_{1i} &\coloneq& \Pb\left(\Delta_i = \max_j \Delta_j\right), 
\mbox{ and} \\ 
q_{2i} &\coloneq& \sum \limits_{j \neq i} \Pb\left(\Delta_i^{(\delta)} > \Delta_j \mid \Delta_i = \max_k \Delta_k \right) \Pb\left(\Delta_i = \max_k \Delta_k\right) + \\ 
& &\Pb\left(\Delta_i > \Delta_j^{(\delta)} \mid \Delta_j = 
\max_k\Delta_k\right) \Pb\left(\Delta_j = \max_k \Delta_k\right).
\end{array}
$}
$$

These represent the probability that moving to node $i$ is the greedy 
choice (corresponding to the highest net gain), or would become the 
greedy choice if the adversary were to perturb the node with the largest
reward, respectively. 

After each step, the traveler observes up to one new node reward, 
$x$.
She uses this to update her probability $\pi$ that the adversary is Type 0.
Assume she begins with a $\text{Beta}(\alpha, \beta)$ prior over probability of the adversary being Type~0.
After observing reward $x$, the posterior mean probability of Type~0 is given by
\begin{align*}
    \gamma &= \Pb\left(\mbox{Type 0} \mid x\right) \\
           &= \frac{f(x \mid \text{Type 0}) \, \mathbb{E}[\pi]}{f(x \mid \text{Type 0}) \, \mathbb{E}[\pi] + f(x \mid \text{Type 1}) (1 - \mathbb{E}[\pi])},
\end{align*}
where $\mathbb{E}[\pi] = \alpha / (\alpha + \beta)$.

To compute the likelihoods in this setting,
we distinguish between observations arising from unreduced and reduced node rewards.
The traveler's subjective $\mathcal{MVT}$ distribution models the true (i.e., unreduced) reward and cost values.
Hence, when a node's reward has been reduced by a known fraction $\delta$,
the observed value must be rescaled (i.e. inflated) by the factor $1/(1 - \delta)$
to be evaluated under the subjective model.
Let $x^* = x /(1-\delta) $.

Based on whether the reward at a node has been reduced (due to sabotage) or remains intact, 
we consider two cases for each adversary type, corresponding to the traveler's observation likelihoods under each scenario.
Take the Type 0 adversary first. 
He always reduces the best node, which is node $i$ with probability 
$q_{1i}$. 
In that case, the conditional density of $x$ is 
$$
    f\left(x \, \mid \mbox{ Type 0}\right) 
    = q_{1i}f\left(x^*\right) + 
    \left(1 - q_{1i}\right)f\left(x\right).
$$

We now consider the Type 1 adversary, which is slightly more complicated. 
The Type 1 adversary reduces the node that has the second-best 
net payoff.
This is node $i$ with probability $q_{2i}$. 
Thus, 
$$
    f\left(x \, \mid \mbox{ Type 1}\right) 
    = q_{2i} f\left(x^*\right) + \left(1 - q_{2i}\right) f\left(x\right).
$$
Putting this together gives the result in \eqref{gamma}. 

\section*{Acknowledgments:}  Elvan Ceyhan and Li Zhou were supported by 
the Office of Naval Research Grant N00014-22-1-2572 and the NSF Award
\#2319157.
David Banks and Leah Johnson were supported by the Office of Naval 
Research Grant N00014-22-1-2572.

\bibliographystyle{plainnat}   
\bibliography{ref.bib}

\end{document}